\documentclass[12pt]{amsart}

\usepackage{amssymb}

\usepackage{enumitem}

\makeatletter
\@namedef{subjclassname@2020}{%
  \textup{2020} Mathematics Subject Classification}
\makeatother

\newtheorem{lem}{Lemma}[section]
\newtheorem{prop}{Proposition}%[section]

\newtheorem{defi}{Definition}[section]
\newtheorem{coro}{Corollary}[section]
\newcommand{\hf}{\hat{f}}

\newtheorem{theo}{Theorem}
\newcommand{\Br}{\Poin}
\newcommand{\Cr}{{\bf Cr}}
\newcommand{\dist}{{\bf dist}}
\newcommand{\diam}{\mbox{diam}\, }
\newcommand{\compose}{\circ}
\newcommand{\dbar}{\bar{\partial}}
\newcommand{\Def}[1]{{{\em #1}}}
\newcommand{\dx}[1]{\frac{\partial #1}{\partial x}}
\newcommand{\dy}[1]{\frac{\partial #1}{\partial y}}
\newcommand{\Res}[2]{{#1}\raisebox{-.4ex}{$\left|\,_{#2}\right.$}}
\newcommand{\sgn}{{\rm sgn}}

\newcommand{\C}{{\bf C}}
\newcommand{\D}{{\bf D}}
\newcommand{\Dm}{{\bf D_-}}
\newcommand{\N}{{\bf N}}
\newcommand{\R}{{\bf R}}
\newcommand{\Z}{{\bf Z}}
\newcommand{\tr}{\mbox{Tr}\,}

\newenvironment{nproof}[1]{\trivlist\item[\hskip \labelsep{\bf Proof{#1}.}]}
{\begin{flushright} $\square$\end{flushright}\endtrivlist}

\newenvironment{block}[1]{\trivlist\item[\hskip \labelsep{{#1}.}]}{\endtrivlist}

\newtheorem{com}{Comment}

\font\mathfonta=msam10 at 11pt
\font\mathfontb=msbm10 at 11pt
\def\Bbb#1{\mbox{\mathfontb #1}}
\def\lesssim{\mbox{\mathfonta.}}
\def\suppset{\mbox{\mathfonta{c}}}
\def\subbset{\mbox{\mathfonta{b}}}
\def\grtsim{\mbox{\mathfonta\&}}
\def\gtrsim{\mbox{\mathfonta\&}}

\newcommand{\ar}{{\bf area}}
\newcommand{\1}{{\bf 1}}
\newcommand{\Bo}{\Box^{n}_{i}}
\newcommand{\Di}{{\cal D}}
\newcommand{\gd}{{\underline \gamma}}
\newcommand{\gu}{{\underline g }}
\newcommand{\ce}{\mbox{III}}
\newcommand{\be}{\mbox{II}}
\newcommand{\F}{\cal{F_{\eta}}}
\newcommand{\Ci}{\bf{C}}
\newcommand{\ai}{\mbox{I}}
\newcommand{\dupap}{\partial^{+}}
\newcommand{\dm}{\partial^{-}}

{\end{enumerate}\end{sc}}
\begin{document}

%%%%% To ease editing, for IMPAN journals add:

\baselineskip=17pt

%%%%%%%%%%%%%%%%

\title{Fixed points of the Ruelle-Thurston operator and the Cauchy transform
%\thanks{Results of this note were reported partially in \cite{lreport}
}

\author{Genadi Levin
%\thanks{Supported in part by ISF grant 1226/17}
}
\address{Institute of Mathematics\\ Hebrew University of Jerusalem\\
Givat Ram\\
Jerusalem, Israel}
\email{levin@math.huji.ac.il}

\date{}

\begin{abstract}
We give necessary and sufficient conditions for a function in a naturally appearing functional space to be a fixed point of the Ruelle-Thurston operator associated to a rational function, see Lemma \ref{Tcauchy}. The proof uses essentially a recent \cite{lll}.
As an immediate consequence, we revisit Theorem 1 and Lemma 5.2 of \cite{Le}, see Theorem \ref{tworatintro} and Lemma \ref{contr} below.
\end{abstract}



\maketitle

\section{Introduction}\label{intro}
Let $f$ be a rational function of degree at least $2$. The Ruelle-Thurston (pushforward) operator $T_f$
associated to $f$ acts on a function $g$ as follows:
$$T_fg(x)=\sum_{w: f(w)=x}\frac{g(w)}{f'(w)^2}$$
whenever the right-hand side is well-defined.
(See \cite{Ruelle}, \cite{dht}, \cite{LSY} and references therein for some background.)
We say that $H:\C\to\C$ which is defined Lebesgue almost everywhere is a fixed point of $T_f$ if
$T_fH(z)=H(z)$ for almost every $z\in\C$.

%The operator $T_f$ appears in several settings.
In \cite{lpade}, see also \cite{LSY}, we calculate the action of the resolvent $(1-\rho T_f)^{-1}$ on the Cauchy kernel to study Pad\'e approximations
to the function $\int d\mu(u)/(z-u)$ in the basin of $\infty$ of a polynomial $f$ where $\mu$ is the equilibrium measure of $f$.
%of $B'/B$ for the Botcher coordinate $B$ of a polynomial $f$ at $\infty$.
%to the Cauchy transform of the equilibrium measure of a polynomial $f$.
On the other hand, $T_f$ appears \cite{dht}
%in Douady-Hubbard's proof of Thurston's topological realisation of rational maps \cite{dht}
in a duality relation $\int_{\C}H(z)(f^\ast\nu)(z)d\sigma_z=\int_{\C}(T_f H)(z)\nu(z)d\sigma_z$
where $\sigma_z$ (here and later on) is the Lebesgue measure on the $z$-plane and $f^\ast:\nu\mapsto \nu\circ f (\bar{f'}/f')$ is the pullback operator that acts on Beltrami coefficients $\nu$. Absence of
non-trivial fixed points of $f^\ast$ supported on $J(f)$ (unless $f$ is a so-called flexible Lattes map) is equivalent to the 'no invariant line fields' conjecture
which, in turn, would imply the fundamental 'density of hyperbolicity', or Fatou conjecture, see e.g. \cite{mcm}.
%The operator $T_f$ turns out to be also closely connected to the above conjectures and related problems.
%On the other hand, absence of non-trivial fixed points of $T_f$ in a 'dual' space of functions (or quadratic differentials) is usually easier to deal with.
The operator $T_f$ has been used by Douady and Hubbard in their proof of Thurston's topological realisation of rational maps \cite{dht}, and then applied to transversality/'no invariant line field'/other problems e.g., \cite{Astorg}, \cite{be}, \cite{e}, \cite{Ep2}, \cite{Le1}, \cite{leij1}, \cite{mult}, \cite{Le}, \cite{LSS}, \cite{mak}, \cite{ts} (see also \cite{vstr}). See an example of application after Theorem \ref{tworatintro} of this paper.
For an alternative (local) approach and discussions, see \cite{lss0}.

A key point in those applications is the fact that $1$ is not an eigenvalue of the operator $T_f$, i.e.,
$T_f$ has no non-trivial fixed points in a relevant space (cf. \cite{lss0}). The
case of meromorphic on the Riemann sphere integrable functions (quadratic differentials) is covered in \cite{dht},
more general forms of fixed points are considered in \cite{e}, \cite{Le1}, \cite{Le}, \cite{mak}, \cite{ts}.
Our main result, Lemma \ref{Tcauchy}, describes the set of the fixed points of the operator $T_f$ as well as give conditions for the triviality of this set
in a rather general and natural space of functions. It includes, for example,
Cauchy transforms of finite discrete measures supported on critical orbits of $f$. Fixed points of $T_f$ of this form, for specific classes of maps $f$,
appear in the above mentioned works.
%, with a finite set support e.g., \cite{dht} as well as infinite one e.g., \cite{ts}, \cite{Le1}.
In the present paper we consider the general situation allowing the postcritical set to intersect boundaries of Herman rings of $f$,
%the case that has been a somewhat elusive ????
the remaining
%somewhat elusive
case that has not been been covered
before, cf. \cite{mak}, \cite{Le}, \cite{Astorg}. This case needs separate special considerations.

Lemma \ref{Tcauchy}, more precisely, its corollary Lemma \ref{contr}, allow us to revisit Theorem 1 of \cite{Le}, see Theorem \ref{tworatintro} below, to cover also maps with Herman rings.
Theorem 1, \cite{Le} and its revision Theorem \ref{tworatintro} of the present note include or imply many of the previous results in this direction, e.g., \cite{ts}, \cite{Le1}, \cite{mak}, \cite{vstr}, \cite{be}, and have found new applications in \cite{gash}, \cite{mag}, \cite{Astorgetc}.

%{\bf Acknowledgments}. We would like to thank Magnus Aspenberg for asking about the proof of Lemma 5.2 in \cite{Le},
%Matthieu Astorg for discussions, and the referee for careful reading of the paper and helpful comments.
\section{Statements and comments}
Given a finite complex measure $\mu$ on $\C$, consider the Cauchy transform of $\mu$:
\begin{equation}\label{cauchydef}
\hat\mu(z)=\int\frac{d\mu(w)}{w-z}.
\end{equation}
%By (\ref{muatinfty}),
The integral converges absolutely Lebesgue almost everywhere and $\mu$ is holomorphic in $\bar\C\setminus supp(\mu)$ (see Sect \ref{cauchy} for more details).
We will assume additionally that $\mu$
satisfies the following condition at $\infty$:
\begin{equation}\label{muatinfty}
\int_{|z|>10}|z|\log|z|d|\mu|(z)<\infty.
\end{equation}
In particular, (\ref{muatinfty}) holds if $\mu$ has a compact support on $\C$. Denote
$$A=\int d\mu(z)=\mu(\C), \ \ \ \ \ \ \ \ B=\int zd\mu(z)$$
existing by (\ref{muatinfty}). Note that $\hat\mu$ is integrable at $\infty$ if and only if $A=B=0$.
%, see Sect \ref{mainlemma}, Step I.
We use the following terminology and notations.
\begin{defi}\label{rset} (cf. \cite{Le})
A compact $E\subset \C$ is a A-compact if $A(E)=R(E)$ where $A(E)$ is the algebra of all continuous function on $E$ which are analytic in the interior of $E$
and $R(E)$ is the algebra of uniform limits on $E$ of rational functions with poles outside $E$ ($=$uniform limits on $E$ of functions holomorphic on $E$).
$E\subset\bar\C$ is a A-compact if $M(E)$ is a
A-compact for some (hence, any) Mobius transformation $M$ such that $M(E)\subset\C$. If an A-compact $E$ is nowhere dense, it is called a C-compact (C=continuous since in this case $A(E)=C(E)$, the set of all continuous functions on $E$).
\end{defi}
Necessary and sufficient conditions for a compact in the plane to be A- or C-compact are given by Vitushkin \cite{Vit}.
In particular \cite{gamelin}, a compact $E$ is a C-compact if the area of $E$ is zero. $E$ is C-(respectively, A-)compact if every point of $E$ (respectively, every point of $\partial E$) belongs
to the boundary of a component of the complement $\C\setminus E$, in particular, if the complement to $E$ consists of a finitely many components or if $E$ is (a closed subset of) the boundary of an open set
(then $E$ is a C-compact).

Recall that the Herman ring $A$ is a periodic component of the Fatou set $F(f)$ of $f$ which is homeomorphic to an annulus. The boundary $\partial A$ of $A$ consists of two connected components.
\begin{defi}\label{defring} Given a closed subset $K$ of the Julia set $J(f)$ of $f$, denote by $\mathcal{H}(K)$ a collection of all Herman rings $A$ of $f$ such that $\partial A\subset K$ and let
$$\mathcal{H}_K=\cup\{\overline A: A\in\mathcal{H}(K)\}.$$
%Similarly, given a sequence ${\bf b}=\{b_k\}_{k=0}^\infty\subset J(f)$ with the closure $K=\overline{\{b_k\}}$, let $\mathcal{H}_{discr}({\bf b})$ be a subcollection of $\mathcal{H}(K)$ as follows: $A\in\mathcal{H}_{discr}({\bf b})$ if and only if there exists a subsequence
%$\{b_{k_n}\}$ which is dense in $\partial A$ (i.e., $\{b_{k_n}\}\subset\partial A$ and $\overline{\{b_{k_n}\}}=\partial A$).
%Denote
%$$\mathcal{H}_{discr,{\bf b}}:=\cup\{\overline A: A\in\mathcal{H}_{discr}({\bf b})\}.$$
\end{defi}
Main result is the following. Remark that conditions on $f$ at $\infty$ as well as (1.1)-(1.2) in the next lemma are served for the proof of Theorem \ref{tworatintro}, as in \cite{Le}.
\begin{lem}\label{Tcauchy}
Let $f$ be a rational function which is not a flexible Lattes map.
%normalized so that $f(z)=\sigma z + b + O(1/z)$ for some $\sigma\neq 0,\infty$, and $f$ is not a flexible Lattes map.

{\bf 1}. Suppose $f$ is normalized so that $f(z)=\sigma z + b + O(1/z)$ for some $\sigma\neq 0,\infty$.
Let $\mu$ be a measure that satisfies (\ref{muatinfty}) such that the function $H(z):=\hat\mu(z)$
is a fixed point of the operator $T_f$.
Assume
%that $$K:=supp(\mu)\subset J(f),$$
that either (1.1) or (1.2) holds:

(1.1) $A=B=0$,

(1.2) either $|\sigma|>1$, or $\sigma^q=1$ for some $q\in\N$ and $b=0$,
or $\sigma=1$ and $A=0$.

Assume that $K:=supp(\mu)\subset J(f)$ and, moreover,

(CL) $K$ is a C-compact.

Then $\mu=0$ outside $\mathcal{H}_K$, i.e., $K\subset\partial\mathcal{H}_K=\cup_{A\in\mathcal{H}(K)}\partial A$.
In particular, $\mu=0$ if $\mathcal{H}_K=\emptyset$.
If $\mu\neq 0$ and, additionally to (CL),

(AL) $\mathcal{H}_K$ is a A-compact,
%Assume that
%\begin{enumerate}
%\item[(CL)] $K$ is a C-compact
%\item[(AL)] $\mathcal{H}_K$ is a A-compact.
%\end{enumerate}

then the following representation holds:
%if $\mu\neq 0$, and the function $H(z):=\hat\mu(z)$ is a fixed point of $T_f$,
%the following must hold:
%$\mathcal{H}(K)$ is a non-empty collection of Herman rings of $f$
%, $supp(\mu)=\cup_{A\in\mathcal{H}(K)}\partial A$
%and
\begin{equation}\label{repr}
\mu=\sum_{A\in\mathcal{H}(K)}\mu_A
\end{equation}
where $\mu_A$ is a measure supported on $\partial A$, $\mu_A$ is absolutely continuous w.r.t. harmonic measure of $A$, at least one of $\mu_A$ is not trivial, and $\mu_A$, $\mu_{A'}$ are mutually singular for $A\neq A'$.
In particular, $\mu$ is non-atomic.
Moreover, if $\mu_A\neq 0$, then $A\in\mathcal{H}(K)$ must satisfy the following property:
if $\psi_A:\Delta_A\to A$ is a holomorphic homeomorphism from a round annulus $\Delta_A$ onto $A$, then $1/\psi_A'$ is in the $H^1$-Hardy space, i.e.,
\begin{equation}\label{hardy}
\limsup_{\epsilon\to 0}\int_{\{z\in\Delta_A: \dist(z,\partial\Delta_A)=\epsilon\}}\frac{|dw|}{|\psi_A'(w)|}<\infty
\end{equation}
if $\mathcal{H}_K\subset\C$ is bounded, and otherwise (\ref{hardy}) holds for $M(A)$ instead of $A$
where $M$ is a Mobius transformation such that $M(\mathcal{H}_K)$ is a bounded subset of $\C$.
%Furthermore, for each $A\in\mathcal{H}$, the Radon-Nikodym derivative $h_A\in L^1(\partial A,\omega_A)$ of $\mu_A$ w.r.t. $\omega_A$,
%is characterized as follows:
%let $\varphi_A: A\to\Delta_A$ is a holomorphic homeomorphism from $A$ onto a round annulus $\Delta_A$
%and $\psi_A=\varphi^{-1}_A:\Delta_A\to A$. Then, for a.e. $w\in\partial\Delta_A$ (w.r.t. the arc measure
%on the boundary circles of $\Delta_A$), $1/\psi_A'\in H^1(\Delta_A)$, i.e.,
%$$\limsup_{\epsilon\to 0}\int_{\{z\in\Delta_A: \dist(z,\partial A)=\epsilon\}}\frac{|dw|}{|\psi_A'(w)|}<\infty$$
%$$h_A(z)=\frac{\varphi_A'(z)^2}{\varphi_A(z)^2}$$
%$$h_A(\psi_A(w))=\frac{1}{w^2\psi_A'(w)^2}.$$
%where, for $\omega_A$-a.e. $z$,
%$$\frac{\varphi_A'(z)}{\varphi_A(z)}=\lim_{z'\to z, z'\in R_z} \frac{\varphi_A'(z')}{\varphi_A(z')}$$
%and the limit exists
%for $\omega_A$-a.e. $z\in\partial A$
%along the ray
%$R_z=\psi_A(\{w'\in A, w'=r w\})$ with $z=\psi_A(w)$, $w\in\partial\Delta_A$.
%(A) $K$ is C-compact and $K$ contains no the boundary of a Herman ring of $f$.
%(B) $f$ has a (maximal non-empty) collection $\mathcal{H}=\{A\}$ of Herman rings such that
%$$\cup_{A\in \mathcal{H}}\partial A\subset K$$
%$K$ is C-compact and, moreover,
%and $K\cup_{A\in \mathcal{H}}\overline{A}$ is A-compact.
%Let $K_{\mathcal{H}}=\cup_{A\in \mathcal{H}}\overline{A}$. Then $K_{\mathcal{H}}$ is H-compact
%and $\tilde K$ is R-compact, where $\tilde K$ is
%the union of those components of $K$ that intersect $\partial{K_{\mathcal{H}}}$. Besides, $K$ (in fact, enough $K\setminus\tilde K$) is C-compact.

{\bf 2}. Vice versa, let $\mathcal{H}=\{A,f(A),\cdots,f^{q-1}(A)\}$ be the cycle of a Herman ring $A$. Assume that
$1/\psi_A'\in H^1$. Then $1/\psi_B'\in H^1$ for every $B\in\mathcal{H}$ and there exists a finite complex measure $\mu\neq 0$ which is supported on $\cup_{B\in\mathcal{H}}\partial B$ such that:
\begin{enumerate}
%, for $H:=\hat\nu$: (2.1) $H=\hat\mu$ is a fixed point of $T_f$
\item [(2.1)] $T_f \hat\mu=\hat\mu$ in $\C\setminus supp(\mu)$ and $\hat\mu=0$ in $\bar\C\setminus\cup_{B\in\mathcal{H}}\overline B$,
\item [(2.2)] the representation (\ref{repr}) holds (with $\mathcal{H}$ instead of $\mathcal{H}(K)$).
\end{enumerate}
The measure $\mu$ is unique in the following sense: if $\nu$ is another measure with the same support $supp(\nu)=\cup_{B\in\mathcal{H}}\partial B$
for which (2.1) holds with $\mu$ replaced by $\nu$, then $\nu=k \mu$ for some constant $k\in\C$.
%and, for each $A\in\mathcal{H}$,
%$H|_{A}=C_A(\varphi'/\varphi)^2$ for some $C_A\in\C$
\end{lem}
Notice that if the boundary curves of a Herman ring $A$ happen to be sooth enough (say $C^2$) then $1/\psi_A'\in H^1$.
%After part {\bf 2}, we can say that a {\it cycle of Herman rings is in} $H^1$ if $1/\psi_A'\in H^1$ for some, hence all, rings $A$ of the cycle.
%\begin{coro}\label{ker}
%Let $f$ be a rational function such that the closure of the union of all its Herman rings is a A-compact.
%Let $X$ be a collection of all functions of the form $\hat\nu$, for some finite complex measures $\nu$ such that $supp(\nu)\subset J(f)$ and $supp(\nu)$ is a C-compact. Then $X$ is a linear space and dimension of the kernel of the operator $T_f$ restricted to the space $X$ is equal to
%the total numbers $\kappa\ge 0$ of cycles $\mathcal{H}_1,\cdots,\mathcal{H}_\kappa$ of Herman rings which are in $H^1$.
%Then Lemma \ref{Tcauchy} gives necessary and sufficient conditions for a function of this space to be a fixed point of $T_f$.
%\end{coro}
%The proof of Lemma \ref{Tcauchy} and Corollary \ref{ker} occupies Sect \ref{mainlemma}.
\begin{com}\label{always}
Presumably, closures of Herman rings are mutually disjoint and the complement to the closure of any Herman ring consists of a finitely many components. That would imply that
%the union of closures of Herman rings of a rational function is always a A-compact, in particular,
the condition (AL) always holds.
Note that (AL) holds if, for example, boundaries of Herman rings of $f$ are locally connected.
%(e.g., Jordan curves) because in this case the complement to $\mathcal{H}_K$ consists of a finitely many components.
\end{com}
\begin{com}
(cf. \cite{Astorg})
Condition (CL) can be replaced by the following one:
\begin{enumerate}
\item[($\widetilde{CL}$)] $f$ carries no an invariant line field on $K$.
%and $\mathcal{C}_\mathcal{H}$ is a A-compact.
\end{enumerate}
This follows at once from Step III of the proof of Lemma~\ref{Tcauchy}, see Sect \ref{mainlemma}.
The condition ($\widetilde{CL}$) in the case when $\mathcal{H}_K$ is empty, i.e., $K$ contains no boundaries of Herman rings was, in fact, observed in \cite{Astorg}.
\end{com}
The proof of Lemma \ref{Tcauchy} goes along the following lines, see Section \ref{mainlemma}. First, using the contraction property of the operator $T_f$,
it is shown that $H=0$ outside $K\cup\mathcal{H}_K$.
%it is noted that $H$ is the Cauchy transform of a discrete measure and
If $K$ contains no boundaries of Herman rings and $K$ is a C-compact, it follows that then $\mu=0$. In fact, in this case the proof is not original and is more or less
minor variation of arguments scattered in~\cite{dht},~\cite{e},~\cite{Le1},
~\cite{mak},~\cite{mult},~\cite{Le}.
%Case (B) is missed in \cite{Le}.
The case that $K$ does contain boundaries of Herman rings is the main content of the present note. In this case, we use some recent results about the Cauchy transform from \cite{lll}: the claim involving (CL) will follow from Lemma 2.1 and involving also (AL) - from Corollary 2.1 of \cite{lll}, we state them in Sect \ref{cauchy} for the reader's convenience.
%and Propositions \ref{remov}-\ref{rings} of the present note, to conclude that $\mu=0$ unless $H$
%is the Cauchy transform of an explicitly described measures as in the conclusion of part {\bf 1} of Lemma \ref{Tcauchy}.

%Given a sequence ${\bf b}=\{b_k\}_{k=0}^\infty\subset\C$ with the closure %$K=\overline{\{b_k\}}\subset\bar\C$, let us
%introduce a subcollection $\mathcal{H}_s({\bf b})$ of $\mathcal{H}(K)$ as follows: $A\in\mathcal{H}_s({\bf b})$ if and only if there exists a subsequence
%$\{b_{k_n}\}$ which is dense in $\partial A$ (i.e., $\{b_{k_n}\}\subset\partial A$ and $\overline{\{b_{k_n}\}}=\partial A$).
%Denote
%$$\mathcal{H}_{s,{\bf b}}:=\cup\{\overline A: A\in\mathcal{H}_s({\bf b})\}.$$
Let us draw a corollary which is suitable for the main application, Theorem \ref{tworatintro}.
Suppose that $V:=\{v_1,\cdots,v_\ell\}\subset J(f)$ is a collection of critical values of $f$.
Let
$$K=\overline{\cup_{j=1}^\ell O^+(v_j)}$$
where $O^+(x)=\{f^i(x)\}_{i\ge 0}$ denotes the forward orbit of a point $x$. Note that $K\subset J(f)$.
\begin{defi}\label{defringcr}
Let $\mathcal{H}_{crit}(V)$ be a subcollection of $\mathcal{H}(K)$
%as follows: $\mathcal{H}_{crit}({\bf v})$
of those Herman rings $A$
%whose (periodic) orbits of those Herman rings $A$
    %=\cup\{\overline A: A\in\mathcal{H}_s({\bf c})\}$ where $A\in\mathcal{H}_s({\bf c})$ if and only if
such that there is a pair of different indices $1\le i<i'\le \ell$ with the property that
$O^+(v_i)$ is a dense subset of $L:=\cup_{k=0}^{q-1}f^k(L_A)$ and
$O^+(v_{i'})$ is a dense subset of $L':=\cup_{k=0}^{q-1}f^k(L_A')$ where $L_A,L'_A$ are two components of $\partial A$ and $q$ is the period of $A$.
Denote
$$\mathcal{H}_{crit,V}:=\cup\{\overline A: A\in\mathcal{H}_{crit}(V)\}.$$
\end{defi}
Now, suppose that, as in Lemma \ref{Tcauchy}, $f$ is not a flexible Lattes map with the normalization $f(z)=\sigma z+b+O(1/z)$ at $\infty$.
%For each $j\in\{1,\cdots,\ell\}$, let $m_j$ be a (discrete, finite, complex) measure supported in $O^+(v_j)\cap\C$.
%such that $m_j(\{v_j\})\neq 0$ whenever $v_j$ is not periodic.
%Let us apply Lemma \ref{Tcauchy} to a measure $m$ which is a linear combination of $m_1,\cdots,m_\ell$.
%$$m=\sum_{k\ge 0}\alpha_k \delta_{b_k}, \ \ \ \ \ \ \ \alpha_k, b_k\in\C,$$
%where $\sum_{k\ge 0} |\alpha_k|(1+|b_k|^2)<\infty$. In particular,
%$\sum_{\{|b_k|>10\}}|\alpha_k||b_k|\log|b_k|<\infty$, i.e., the condition (\ref{muatinfty}) holds for $m$,
%and $A=\sum_{k\ge 0} \alpha_k$, $B=\sum_{k\ge 0} \alpha_k b_k$ exist.
%Let ${\bf b}=\{b_k\}$ and $K=\overline{\{b_k\}}=supp(m)$.
%In this case, introduce a subcollection $\mathcal{H}(K)_s$ of $\mathcal{H}(K)$ as follows: $A\in\mathcal{H}(K)_s$ if and only if there exists a subsequence
%$\{b_{k_n}\}$  which is dense in the boundary of $A$ (in particular $\{b_{k_n}\}\subset\partial A$).
%Denote
%$$\mathcal{H}_{K,s}:=\cup\{\overline A: A\in\mathcal{H}(K)_s\}.$$
%(Note that we don't assume that all $b_k$'s are pairwise different.)
\begin{lem}\label{contr} (cf. \cite{Le}, Lemma 5.2)
For each $j\in\{1,\cdots,\ell\}$, let $m_j$ be a discrete, finite, complex measure supported in $O^+(v_j)\cap\C$, and $m$ is a linear combination of $m_1,\cdots,m_\ell$.
Let
\begin{equation}\label{H}
H(x):=\hat m(x)=\sum_{k=0}^\infty \frac{\alpha_k}{b_k-x}. \ \mbox{ Assume } \ \sum_{k\ge 0} |\alpha_k|(1+|b_k|^2)<\infty.
\end{equation}
In particular, $A=\sum_{k\ge 0} \alpha_k$, $B=\sum_{k\ge 0} \alpha_k b_k$ exist.
%Let $K=supp(m)$, i.e., $K$ is the closure (on the Riemann sphere)
%of the set $\{b_k\}$.
Assume that $m_i(\{v_i\})\neq 0$ whenever $v_i$ is neither periodic nor in the forward orbit of any other $v_{i'}$, $i'\neq i$.
Assume, additionally, that the following conditions of the part {\bf 1} of Lemma \ref{Tcauchy} are satisfied: either (1.1) or (1.2) as well as (CL). Finally, assume

(AL$_{cr}$) $\mathcal{H}_{crit,V}$ is a A-compact.

Then $T_f H=H$ implies $m=0$ (i.e., $H=0$).
\end{lem}
Indeed, let by a contradiction $m\neq 0$ so that after perhaps throwing away some indices and re-numerating the rest, one can assume that $m=\sum_{1\le j\le \ell} a_j m_j$ where all $a_j\neq 0$.
As, by (CL), $K$ is a C-compact, by Lemma \ref{Tcauchy}, part {\bf 1}, $supp(m)\subset\partial\mathcal{H}_K$, in particular, one can assume from the beginning that
each $O^+(v_j)$ falls into the boundary of some $A\in\mathcal{H}(K)$.
%, moreover, $A\in\mathcal{H}(K)$ if and only if $\partial A=\overline{\partial A\cap\cup_{1\le j\le \ell}O^+(v_j)}$.
%In fact, already any $v_j$ is in the boundary of some Herman ring. Indeed,
If $v_j$ is periodic, then this obviously implies that $v_j$ is in the boundary of some Herman ring.
If $v_j$ is not periodic, then $v_j$ is in the forward orbit of some $v_i$ which is neither periodic nor in the forward orbit of any other $v_{i'}$. Hence, $m(\{v_i\})=a_i m_i(\{v_i\})\neq 0$ which implies that again $v_i$, hence, $v_j$ is in the boundary of some Herman ring. This proves that any $v_j$ is in the boundary of some Herman ring. Therefore, $K=\overline{\cup_{j=1}^\ell O^+(v_j)}$ is a subset of boundaries of Herman rings, hence,
by the definition of $\mathcal{H}(K)$, the union $\cup_{1\le j\le \ell}O^+(v_j)$ is a dense subset of
$\partial\mathcal{H}_K$.
Having that, the following {\bf Claim} shows that $\mathcal{H}_{crit,V}=\mathcal{H}_K$, hence, by (AL$_{cr}$), that $\mathcal{H}_K$ is a A-compact. Therefore, by Lemma \ref{Tcauchy}, part {\bf 1}, $m$ has no atom, a contradiction.

{\bf Claim}.
{\it Let $x$ be in a component $L_A$ of the boundary of a Herman ring $A$. Then $O^+(x)$ is either nowhere dense or (everywhere) dense in $L=\cup_{k=0}^{q-1}f^k(L_A)$, $q$ is the period of $A$}.

Indeed, assume without loss of generality that $q=1$ and that there is a ball $B$ centered at a point of $L$($=L_A$) such $B\cap L=B\cap \omega(x)$. As the harmonic measure of $B\cap L$ is positive, there is a subset $X_B$ of positive Lebesgue length in one of the boundary circle $S$ of the annulus
$\Delta_A$ which uniformizes $A$ such that for each $w\in X_B$ there exists the radial limit $\psi_A(w)$ of $\psi_A:\Delta_A\to A$ and $\psi_A(w)\in B$. Hence, the
set $X=\cup_{i\ge 0}\lambda^i X_B$, for the corresponding $\lambda\in S^1$ where $\psi_A^{-1}\circ f\circ\psi_A: w\mapsto\lambda w$, has the full length in $S$ and, for each $w\in X$, $\psi_A(w)$ exists and is in $\omega(x)$. Therefore, $\omega(x)$ is a closed and dense subset of $L$, i.e., $\omega(x)=L$. Thus either $\omega(x)$ is nowhere dense in $L$ or is equal to $L$.

\

The proof of the following Theorem~\ref{tworatintro} is (literally) identical to the proof of Theorem 1 of \cite{Le} (see Sect 5.3 there), after replacing Lemmas 5.2-5.3 of \cite{Le} by Lemma~\ref{contr}.
Recall that a critical point $c$ of $f$ with the forward orbit in $\C$ is called summable if, for $v=f(c)$,
$\sum_{n=0}^\infty\frac{1+|f^n(v)|^2}{1+|v|^2}\frac{1}{|(f^n)'(v)|}<\infty$.
\begin{theo}\label{tworatintro}(cf. \cite{Le}, Theorem 1)
%, \cite{Astorg})
Let $f$ be an arbitrary rational function of degree $d\ge 2$ which is not a flexible Lattes example.
Suppose that $\{c_1,...,c_r\}$ is a collection
of $r$ summable critical points of $f$,
and the union
$$\mathcal{C}:=\cup_{j=1}^r \omega(c_j)$$
of their $\omega$-limit sets
satisfies the following conditions:
\begin{enumerate}
\item[(C)] $\mathcal{C}$ is a C-compact,
\item[(A$_{cr}$)] $\mathcal{H}_{crit,{\bf v}}$ is a A-compact where ${\bf v}:=\{f(c_1),\cdots,f(c_r)\}$.
\end{enumerate}
Replacing if necessary $f$ by its equivalent
(i.e., Mobius conjugate),
one can assume the forward orbits
of $c_1,...,c_r$ avoid infinity.
Consider the set $X_f$ of all rational functions of degree $d$
which are close enough to $f$ and have the same number $p'$
of different critical points with the same corresponding multiplicities.
Then there is a $p'$-dimensional manifold $\Lambda_f$ and its
$r$-dimensional submanifold $\Lambda$, $f\in \Lambda\subset \Lambda_f
\subset X_f$, with the following properties:

(a) every $g\in X_f$ is equivalent
to some $\hat g\in \Lambda_f$,

(b) for every one-dimensional family $f_t\in \Lambda$ through $f$,
such that $f_t(z)=f(z)+t u(z)+O(|t|^2)$ as $t\to 0$,
if $u\not=0$, then,
for some $1\le j\le r$, the limit
\begin{equation}\label{familyratintro}
\lim_{m\to \infty}\frac{\frac{d}{dt}|_{t=0}f_t^m (c_j(t))}
{(f^{m-1})'(f(c_j))}=\sum_{n=0}^\infty \frac{u(f^n(c_j))}{(f^{n})'(f(c_j))}\neq 0
\end{equation}
exists and is a {\it non-zero} number.
Here $c_j(t)$ is the critical point of $f_t$, such that $c_j(0)=c_j$.
%and $v_j=f(c_j)$.
%Furthermore, if $f$ and all the critical values of $f$ are real, then the conditions (C)-(A) hold and the above maps and spaces can be taken real.
\end{theo}
Note the following particular case: all points $c_1,\cdots,c_r$ belong to the same grand orbit i.e., $f^{n_i}(c_i)=f^{n_j}(c_j)$ for any $i,j$ and some $n_i,n_j>0$. Then by Definition \ref{defringcr}, $\mathcal{H}_{crit,{\bf v}}$ is empty, hence, the condition (A$_{cr}$) is void:

{\bf Complement to Theorem \ref{tworatintro}}. {\it If all points $c_1,\cdots,c_r$ are in the same grand orbit (e.g., if $r=1$ or $r>1$ but $f(c_1)=f(c_2)=\cdots=f(c_r)$), then the condition (A$_{cr}$) is unnecessary.}

In Theorem 1 of \cite{Le}, the condition (A$_{cr}$) of the present Theorem \ref{tworatintro} was absent.
%Compared to Theorem 1 of \cite{Le}, in this Theorem \ref{tworatintro} the condition (A)
%is added.
Presumably, (A$_{cr}$) holds always, see Comment \ref{always}.
%Conditions (C) and (A) are independent. (A) is void if and only if $\mathcal{C}$ contains no boundaries of Herman rings.
%, i.e., $\mathcal{H}(\mathcal{C})$ and $\mathcal{H}_\mathcal{C}$ are empty.
%Obviously, if $\hat\mathcal{C}=\mathcal{C}$, the condition (A) turns into: $\mathcal{C}$ is a C-compact.
The case $\mathcal{H}(\mathcal{C})=\emptyset$
%under the condition (C)
was covered in \cite{Le} (see also \cite{Astorg}). Here we treat the missing case, i.e, when %$\mathcal{H}(\mathcal{C})\neq\emptyset$)
boundaries of some Herman rings are contained in $\mathcal{C}$.
%We use \cite{lll}, see Sect \ref{cauchy} below.

Condition (\ref{familyratintro}) is equivalent to the following: for a coordinate system $\{x_1,\cdots,x_{p'}\}$ in $\Lambda_f$, the rank of the matrix $\{L(c_j,x_k)\}_{1\le k\le p', 1\le j\le r}$,
where
$$L(c_j,x_k)=\lim_{m\to \infty}\frac{\frac{\partial g^m (c_j)}{\partial x_k}|_{g=f}}
{(f^{m-1})'(f(c_j))}$$ is maximal, i.e., equal to $r$.
As mentioned in Section \ref{intro}, Theorem \ref{tworatintro} covers many previous results in this direction (as well as finds new applications).
For example, consider the unicritical family $f_v(z)=z^d+v$. If
$0$ is a summable critical point of $f_{v_0}$, then, by Corollary \ref{c0}(2) below, Theorem \ref{tworatintro} applies, hence, (\ref{familyratintro}) holds and it turns into
$$\lim_{m\to \infty}\frac{\frac{d}{dv}|_{v=v_0}f_v^m (0)}
{(f^{m-1})'(v_0)}=\sum_{n=0}^\infty \frac{1}{(f^{n})'(v_0)}\neq 0.$$
This implies, in particular, that $f_{v_0}$ is unstable in the family $\{f_v\}_{v\in\C}$. Indeed, if $f_{v_0}$ were structurally stable in $\{f_v\}$, then the sequence of functions $\{f_v^m(0)\}$ near $v=v_0$ (hence, the sequence of derivatives $\{\frac{d}{dv}|_{v=v_0}f_v^m (0)\}$, too) would be
bounded, which, along with $(f^{m-1})'(v_0)\to\infty$ would imply that
$\frac{\frac{d}{dv}|_{v=v_0}f_v^m (0)}{(f^{m-1})'(v_0)}\to 0$, a contradiction, see \cite{Le1} for details.

More generally, let $f$ be a rational function (not a flexible Lattes map) with a summable critical point $c$ such that $\omega(c)$ is a C-compact. Then, by the {\bf Complement to Theorem \ref{tworatintro}} (with $r=1$), the inequality (\ref{familyratintro}) must hold, for some family $f_t(z)=f(z)+t u(z)+O(|t|^2)\in \Lambda$
with $u\not=0$. On the other hand, if we assume that $f$ is structurally stable in the (one-dimensional) space $\Lambda$ then repeating the argument for the unicritical family, the limit
in (\ref{familyratintro}) must be zero. This contradiction shows that such $f$ is unstable in $\Lambda$, therefore, also in a bigger space $\Lambda_f$, see Corollary 1.2 of \cite{Le} for more details. In particular, $f$ is unstable in the space of all rational maps of the same degree \cite{mak}.
%Indeed, for a structurally stable $f_{v_0}$, $v_0$ belongs to the interior of the Mandelbrot set, hence, the sequence $\{f_{v}^n(0)\}_{n\ge 0}$ %is uniformly bounded near $v_0$ which implies that
%$\frac{d}{dv}|_{v=v_0}f_t^m (0)}$ is a bounded sequence, a contradiction with

\begin{com}\label{herbt}
%A few remarks about the conditions (C)-(A$_{di}$) are in place.
\begin{itemize}
\item since $\mathcal{C}$ is closed and forward invariant, either $\mathcal{C}$ is nowhere dense or
$\mathcal{C}=\bar C=J(f)$, hence, under the condition (C), $\mathcal{C}$ has to be nowhere dense,
\item $\mathcal{C}$ is a C-compact if and only if $\omega(c_j)$ is a C-compact for every $j\in\{1,\cdots,r\}$.
%moreover, $\omega(c_j)$ is a C-compact if and only if $\overline{O^+(c_j)}$ is C-compact.
If $c_j\in\partial U$ where $U$ is a component of the Fatou set of $f$ (say, an iterate of $c_j$ is in the boundary of a Herman ring), then $\omega(c_j)\subset\cup_{n\ge 0}f^n(\partial U)$, hence
$\omega(c_j)$ is a C-compact (we use that $U$ is (pre-)periodic, by Sullivan's no wandering domain theorem),
%\item
%$\mathcal{H}_{discr,{\bf c}}$ is the union of closures of periodic orbits of those Herman rings $A$
    %=\cup\{\overline A: A\in\mathcal{H}_s({\bf c})\}$ where $A\in\mathcal{H}_s({\bf c})$ if and only if
%such that there is a pair of different indeces $1\le i<i'\le r$ with the property that
%$\{f^k(c_i)\}_{k\ge k_0}$ is dense in $L:=\cup_{k=0}^{q-1}f^k(L_A)$ and
%$\{f^k(c_{i'})\}_{k\ge k_0'}$ is dense in $L':=\cup_{k=0}^{q-1}f^k(L_A')$ where $L_A,L'_A$ are two components of $\partial A$ and $q$ is the period of $A$.
%By Definition \ref{defringcr}, $\mathcal{H}_{crit,{\bf v}}$ is empty, hence, the condition (A$_{cr}$) is void if all points $c_1,\cdots,c_r$ belong to the same grand orbit i.e., $f^{n_i}(c_i)=f^{n_j}(c_j)$ for any $i,j$ and some $n_i,n_j>0$ (e.g., (A$_{cr}$) is void if $r=1$ or if $r>1$ and $f(c_1)=f(c_2)=\cdots=f(c_r)$).
\end{itemize}
\end{com}
Let us list some classes of rational maps $f$ and corresponding sets $\mathcal{C}$ for which the conclusion of Theorem 1 of \cite{Le}(=conclusion of Theorem~\ref{tworatintro} of this note) holds:
\begin{coro}\label{c0}
The conclusion of Theorem~\ref{tworatintro} holds whenever $f$ is not a flexible Lattes and one of (1)-(8) takes place:
\begin{enumerate}
\item[(1)] $f$ has no Herman rings and $\mathcal{C}$ is a C-compact,
\item[(2)] $f$ is a polynomial,
\item[(3)] $J(f)=\bar\C$, and $\mathcal{C}$ is a C-compact,
%(4) All critical values of $f$ which are in $J(f)$ are summable (e.g., Collet-Eckmann), their $\omega$-limit sets are nowhere dense, and $f$ has no neutral cycles,
%(4) $f$ is structurally stable,
\item[(4)] $\mathcal{C}\neq\bar\C$ and the complement to $\mathcal{C}$ consists of a finitely many components,
%(in particular, this holds if $\mathcal{C}$ is totally disconnected),
%and $K$ is C-compact (e.g. of area zero),
\item[(5)] note two particular cases of (4):  (i) $\mathcal{C}$ is totally disconnected, for example, finite, (ii) $\mathcal{C}$
lies in a finite union of disjoint Jordan curves in $\bar\C$, for example, $\mathcal{C}\subset\R$,
\item[(6)] $f: \mathcal{C}\to \mathcal{C}$ is expanding (e.g., the critical points $c_1,\cdots,c_r$ satisfy Misiurewicz's condition),
%(7) $\chi_+(b_n)>0$ for some $n\ge 0$,
\item[(7)] the following two conditions hold:
\begin{enumerate}
\item[(7.1)] the area of $\mathcal{C}\setminus\partial\mathcal{H}_{crit,{\bf v}}$ is zero,
\item[(7.2)] either $\mathcal{H}_{crit,{\bf v}}$ is empty (e.g., all $c_1,\cdots,c_r$ are in a single grand orbit), or the boundary of every Herman rings of $f$ is locally connected.
    %$\mathcal{C}_\mathcal{H}$ is either a A-compact,
%the union $\mathcal{R}$ of the closures of Herman rings has at most countable inner boundary
%(i.e., except of at most countable set, every point of $\partial\mathcal{R}$ belongs to a component of the complement $\bar\C\setminus\mathcal{R}$),
\end{enumerate}
%and the union $\mathcal{R}$ of the closures of Herman rings has at most countable inner boundary
%(i.e., except of at most countable set, every point of $\partial\mathcal{R}$ belongs to a component of the complement $\bar\C\setminus\mathcal{R}$).
\item[(8)] All critical values of $f$ which are in $J(f)$ are summable. Here either $\mathcal{C}$ is nowhere dense or $\mathcal{C}=\bar\C=J(f)$.
%their $\omega$-limit sets are nowhere dense, and $f$ has no neutral cycles,
\end{enumerate}
\end{coro}
Corollary~\ref{c0} follows immediately from Corollary \ref{c1} of Lemma~\ref{Tcauchy} which is stated right after the following comment about applications of Theorem \ref{tworatintro}:
%\begin{com}\label{appun}

Corollary~\ref{c0} implies that Corollaries 1.1-1.2 of \cite{Le} and their proofs remain untouched. Indeed, in Corollary 1.1 of \cite{Le} case (6) of the Corollary~\ref{c0} above applies, and in Corollary 1.2,\cite{Le},
Complement to Theorem \ref{tworatintro} (with $r=1$) applies (see discussion after Theorem \ref{tworatintro}).
%the last part of Comment \ref{herbt}.
By similar reasons, applications of Theorem 1 of \cite{Le} in \cite{gash} and in \cite{mag} are unaffected as well: in \cite{gash} $f$ is a polynomial (so case (2) of Corollary~\ref{c0} applies) and in \cite{mag} $f$ is expanding on $\mathcal{C}$, i.e., case (6) of Corollary~\ref{c0} works (in fact, the case (1) applies as well because $f$ as in \cite{mag} cannot have Herman rings).
%Finally, Remark 5.1 of \cite{LSS} is not affected, too, because in that case the corresponding set $\mathcal{C}$ of Lemma~\ref{contr}, see below, is finite: in the notations of \cite{LSS}, $K=\cup_{1\le k\le N}\{f^r(c_{i_k}): 1\le r\le m_k-1\}\cup\{f^s(c_{j_k}): 1\le s\le n_k-1\}$.
%(Note also that, ultimately, Remark 5.1 of \cite{LSS} is not a part of the proof of \cite{LSS} at all.)
%\end{com}

%Corollary~\ref{c0} follows immediately from the next consequence of Lemma~\ref{Tcauchy}:
Now
\begin{coro}\label{c1}
In the notations of Lemma \ref{Tcauchy}, $T_fH=H$ (where $H=\hat\mu$ with $supp(\mu)\subset J(f)$) implies $\mu=0$ if $f$ is not a flexible Lattes and
one of the conditions (1)-(8) of Corollary~\ref{c0} holds with the following obvious changes:
$\mathcal{C}$ should be replaced by $K\subset J(f)$ and, additionally, in cases (1)-(7),
$K$ is nowhere dense, and in case (7) $\mathcal{H}_{crit,(\bf v})$ is replaced by $\mathcal{H}_K$. In the case (8), $K$ can have interior points.
\end{coro}
\begin{proof} We handle here cases (1)-(7); for the case (8), see the end of Section \ref{mainlemma}.
%[of Corollary~\ref{c1}.]
In cases (1)-(7), $K\subset J(f)$ is a C-compact. Indeed, in cases (1),(3) it is a condition,
in cases (2) and (4), every point of $K$ belongs to a component of the complement, and in cases (6),(7), $K$ is of measure zero:
%this is obvious in case (6) and is well-known in case (7) (note that $K$ is disjoint with the set of critical points of $f$ in this case and, in particular, $J(f)\neq\bar\C$);
%in case (4), $K$ is also C-compact because the Lebesgue measure of $K$ is zero???????????????, see the Introduction of \cite{Le} for references.
It remains to note the following.
%If $f$ is a polynomial, then $f$ has no Herman rings.
In cases (1)-(3), $f$ have no Herman rings.
%as well: this is obvious in (3) and follows from \cite{R-L} in case (4). If (6) holds, then clearly $K$ cannot contain boundaries of Herman rings.
In case (4), the complement to $K$ as well as to $\mathcal{H}_K$
consist of finitely many components, therefore, $K$ is a C-compact (being also nowhere dense) and $\mathcal{H}_K$ is a A-compact.
%if $K$ contains no ???????????????????????????????????? boundaries of Herman rings, then we are in
%the case (A) of Lemma~\ref{contr}. On the other hand, if $K$ does contain boundaries of Herman rings $A$ of????????????????????????????????? a collections $\mathcal{H}$ then each $A\in\mathcal{H}$ is a component of the complement to $K$ and since the total number of such components is bounded hence each point of the boundary of
%$K\cup_{A\in\mathcal{H}}\overline A$ is a boundary points of a component of the complement to $K\cup_{A\in\mathcal{H}}\overline A$, hence, the latter set is a A-compact. Thus we are in the case (B) of
%Lemma~\ref{contr}.
%(6)-(8): since (6) implies (7), it is enough to prove only (7). But by the lemma, either (7) or (8) implies that (B) cannot hold.
%\end{com}
It remains to consider cases (6)-(7).
In case (7), if the boundary of every Herman ring is locally connected it is easy to see that
the closures of two different Herman rings are disjoint and the complement to the closure of every Herman ring consists of a finitely many components. Therefore, $\mathcal{H}_K$ is a A-compact.
%assumptions are that the area of $K\setminus\hat K$ is zero and $K_{\mathcal{H}}$ is a A-compact. As it is shown in the proof of Lemma \ref{contr}, given $x\in K\setminus\hat K$, $H=0$ in a neighborhood $U$ of $x$ and since the area of $K\setminus\hat K$ is zero, $\int_U |H|dxdy=0$.
%This implies that there are no points $b_k$ in $U$, hence, in fact, $\hat K=K_{\mathcal{H}}$, and Lemma \ref{contr} applies.
Finally, as for the case (6), if $f:K\to K$ is expanding then $K$ cannot contain the boundary of a Herman ring as it is shown in the next lemma.
\end{proof}
\begin{lem}
Let $\Omega$ be a Herman ring of a rational function $f$ which is invariant by $f$. Then, for either component $L$ of $\partial L$,
$f:L\to L$ is not expanding.
\end{lem}
\begin{proof} Assume the contrary. Then, passing to an iterate, one can find a neighborhood $U$ of $L$ such that $|Df(z)|>2$ for any $z\in U$.
Fix a conformal homeomorphism $\psi:\{1<|w|<r\}\to \Omega$ so that $f=\psi R\psi^{-1}$ where $R(w)=\lambda w$ is an irrational rotation. We can suppose that $L=\cap_{1<r'<r}\overline{\psi(\{1<|w|<r'\})}$. Fix $1<\rho<r$ so that $U_0:=\psi(\{1<|w|\le \rho\})$ is compactly contained in $U$.
Let $\Gamma=\psi(\{|w|=\rho\})$ and $h>0$ is the distance between two disjoint compact sets $L$ and $\Gamma$. Find $u\in L$ and $v\in \Gamma$ so that $h=|u-v|$. The interval $(u,v]$ must be a subset of $U_0$
because otherwise there would exist $u'\in L$, $v'\in \Gamma$ with $|u'-v'|<h$. Let $g:=\psi R^{-1}\psi^{-1}$ the branch of $f^{-1}$ leaving $\Omega$ invariant and let $\gamma=g((u,v])$. Then $\gamma$ is a semi-open
curve in $U_0$ which begins at $v_{-1}=g(v)$ and tends to $L$. Since $|Df(x)|>2$ for all $x\in\gamma$, the length of $\gamma$, $l(\gamma)<h/2$. As $v_{-1}\in\Gamma$ and $\gamma$ joins $v_{-1}$ and $L$, we arrive at a contradiction with the definition of $h$.
\end{proof}
%{\bf Acknowledgments}. We would like to thank Magnus Aspenberg for asking about the proof of Lemma 5.2 in \cite{Le},
%and Matthieu Astorg for discussions. We thank the referee for careful reading of the paper and helpful comments.
%, Feliks Przytycki and Alex Eremenko for discussions and Alexander Volberg for Example \ref{sv}.
\section{The Cauchy transform of measures}\label{cauchy}
Given a finite complex measure $\nu$ with $supp(\nu)\subset\C$, let
%a compact support on $\C$, let
$$\hat\nu(z)=\int\frac{d\nu(w)}{w-z}$$
be the Cauchy transform of $\nu$. For the following facts,
see e.g. \cite{garnett}. As $\nu$ is finite, by Fubini's theorem, $\int\frac{d|\nu|(w)}{|w-z|}$ (hence, $\hat\nu$) is locally in $L^1(dxdy)$. In particular, $\hat\nu$ exists almost everywhere on $\C$. Besides, $\hat\nu$ is holomorphic outside
of $supp(\nu)$, and $\nu(\infty)=0$ if $\nu$ has a compact support. Moreover, $\hat\nu\neq 0$ on a set of a positive area unless $\nu=0$.
The following two propositions are the main auxiliary statements we use. They appear
in \cite{lll} as Lemma 2.1 and Corollary 2.1, respectively.
\begin{prop}\label{remov}
\begin{enumerate}
\item [(a)] Any closed subset of a C-compact is C-compact.
\item [(b)]
Let $K$ be a nowhere dense compact in $\C$ and $\mu$ a measure on $K$. Suppose that for a neighborhood $W$ of a point $x\in K$, $K\cap\overline W$ is a C-compact and $\hat\mu=0$ on $W\setminus K$. Then $\mu$ vanishes on
$K\cap W$, i.e., $|\mu|(W)=0$.
\end{enumerate}
\end{prop}
\begin{prop}\label{rings}
Suppose $\mathcal{H}$ is a non empty collection of bounded rotation domains
of a rational function $f$.
Let $V=\cup\{A: A\in\mathcal{H}\}$, $E\subset\C\setminus V$ a nowhere dense compact subset such that
$\partial V\subset E$, and $\nu$ be a bounded complex measure supported on $E$ such that $\hat\nu=0$ off $E\cup V$. If $E$ is a C-compact and $\overline{V}$ is a A-compact, then $\nu$ is, in fact, supported on $\partial V=\cup_{A\in\mathcal{H}}\partial A$, $\nu|_{\partial A}$, $A\in\mathcal{H}$, are mutually singular and, for each  $A$, $\nu|_{\partial A}$ is absolutely continuous w.r.t. harmonic measure of $A$. In particular, $\nu$ is non-atomic.
Moreover, for each $A\in\mathcal{H}$,
%the function $\hat\nu\circ\psi_A'$ is in the $H^1$-Hardy space, i.e.,
$$\limsup_{\epsilon\to 0}\int_{\{z\in\Delta_A: \dist(z,\partial \Delta_A)=\epsilon\}}|\hat\nu\circ\psi_A(z)||\psi_A'(z)||dz|<\infty,$$
where $\psi_A:\Delta_A\to A$ is a holomorphic homeomorphism from a round annulus $\Delta_A$ onto $A$.
%Let $V$ be an open set such that each component $\Omega_i$ of $V$ is finitely connected and the collection $\{\Omega_i\}$ is a D-collection. Assume that $A(V)=A(\overline{V})=R(\overline{V})$. If the Cauchy transform of a measure $\nu$ on $\partial V$ vanishes off $\overline{V}$ then $\nu$ is supported on $\cup_i\partial\Omega_i$ and admits the representation (\ref{nurepres}).
%In particular, let $\Omega$ be a finitely connected domain without isolated points in $\partial\Omega$ such that a holomorphic isomorphism $\psi:\Delta\to\Omega$, for a circular domain $\Delta$, extends
%one-to-one on a set of full (arc) measure on $\partial\Delta$. Assume $A(\Omega)=A(\overline\Omega)=R(\overline\Omega)$.
%If the Cauchy transform of a measure $\nu$ on $\partial\Omega$ vanishes off $\overline{\Omega}$ then $\nu$ is absolutely continuous w.r.t. harmonic measure on $\partial\Omega$.
\end{prop}
\section{Proof of Lemma~\ref{Tcauchy}}\label{mainlemma}
Let us start with the part {\bf 1}.
We split the proof into Steps I-VIII. Note that most of arguments are not original and included for completeness. Namely, Step I is taken from the proof of Lemma 5.2, \cite{Le} and along with Steps II-V and the first claim of Lemma~\ref{vi} of Step VI are indeed minor modifications of~\cite{dht},~\cite{e},~\cite{Le1},~\cite{mak},~\cite{mult}. After that, the proof when $K$ contains no boundaries of Herman rings is straightforward, see Step VII.
In a short step VIII we deal with the general case applying Propositions~\ref{remov}-\ref{rings}.

We prove the part {\bf 1} by a contradiction. So let $H$ be the Cauchy transform of a finite complex measure $\mu$ that satisfies (\ref{muatinfty}).
Assume that $\mu\neq 0$, $T_fH=H$, and the conditions of Lemma~\ref{Tcauchy}, part {\bf 1} hold.

{\bf I}. The function
$$\tilde{H}(z):=H(z)+\frac{A}{z}+
\frac{B}{z^2}=
\int[\frac{1}{w-z}+\frac{1}{z}+
\frac{w}{z^2}]d\mu(w)=\int\frac{w^2}{z^2(w-z)}d\mu(w)$$
is integrable at $\infty$.
Indeed,
for every $w$, the function (of $z$)
$w^2/[z^2(w-z)]$ is integrable at $\infty$, and one can write:
$$\int_{|z|>1} \left|\frac{w^2}{z^2(w-z)}\right|d\sigma_z
=|w|\int_{|u|>1/|w|} \left|\frac{1}{u^2(1-u)}\right|d\sigma_u\le $$
%$$|b_n|\left(C_1+\int_{2>|x|>r/|b_n|} \frac{1}{|x|^2}d\sigma_x\right)
%\le |b_n|\left(C_2+2\pi \ln\frac{|b_n|}{r}\right)\le
$$C_1|w|(1+\ln^+|w|),$$
where $\ln^+|w|=\max\{0,\ln|w|\}$.
%Now, for every $R$ big enough, and using the
%condition (\ref{muatinfty}), we have:???????????????????????
Hence,
\begin{eqnarray*}
\int_{|z|>1}|\tilde{H}(z)| d\sigma_z\le
\int d|\mu|(w)\int_{|z|>1}
\left|\frac{w^2}{z^2(w-z)}\right|d\sigma_z\le
C_1\int|w|(1+\ln^+|w|)d|\mu|(w)<\infty,
\end{eqnarray*}
by the condition (\ref{muatinfty}).
%$H$ is analytic in each component of $K^c:=\C\setminus K$.
Now, take $R$ big enough and consider the disk $D(R)=\{|x|<R\}$.
We claim that
\begin{equation}\label{limr2}
\limsup_{R\to \infty} \left\{\int_{f^{-1}(D(R))} |H(x)|d\sigma_x-
\int_{D(R)} |H(x)|d\sigma_x\right\} \le 0.
\end{equation}
Indeed, in the case (1.1),i.e., if $A=B=0$, this follows at once from
the integrability of $H$ at $\infty$.
In the case (1.2),
the conditions on $\sigma$ imply that there is $a>0$, such that
\begin{equation}\label{f-1}
f^{-1}(D(R))\subset D(R+|b|+\frac{a}{R})
\end{equation}
(actually, $f^{-1}(D(R))\subset D(R)$, if $|\sigma|>1$).
On the other hand,
\begin{equation}\label{limr}
\lim_{R\to \infty} \int_{R<|x|<R+|b|+a/R} |H(x)|d\sigma_x=0
\end{equation}
since
$\tilde{H}(x)=H(x) + \frac{A}{x} + \frac{B}{x^2}$
is integrable at $\infty$,
%$\lim_{R\to \infty} \int_{R<|x|<R+|b|+a/R} |\hat H(x)|d\sigma_x=0$.
%But
and an easy calculation shows that the conditions in the case (1.2)
guarantee that
\begin{equation}
\lim_{R\to \infty} \int_{R<|x|<R+|b|+a/R} \left|\frac{A}{x}+\frac{B}{x^2}\right|d\sigma_x=0.
\end{equation}
%so (\ref{limr}) is proved. This,
(\ref{limr}) along with~(\ref{f-1}) imply
(\ref{limr2}) in the case (1.2).
As in the proof of Lemma 5.2, \cite{Le}, (\ref{limr2}) implies that
\begin{equation}\label{tri1}
|H(x)|=|T_fH(x)|=|\sum_{w: f(w)=x}\frac{H(w)}{f'(w)^2}|=\sum_{w: f(w)=x}\frac{|H(w)|}{|f'(w)|^2}
\end{equation}
almost everywhere.
Indeed, otherwise there is a
set $A\subset D(R_0)$ of positive measure (for some $R_0$)
and $\delta>0$, such that
$|T_fH(x)|<(1-\delta)\sum_{w: f(w)=x}\frac{|H(w)|}{|f'(w)|^2}$ for $x\in A$.
Then,
for all $R>R_0$,
$$\int_{D(R)} |H(x)|d\sigma_x=
%\int_{D(R)\setminus A} |H(x)|d\sigma_x+\int_{A} |H(x)|d\sigma_x=
%&=& \\
\int_{D(R)\setminus A} |T_fH(x)|d\sigma_x
+\int_{A} |T_fH(x)|d\sigma_x<$$
$$\int_{f^{-1}(D(R)\setminus A)} |H(x)|d\sigma_x+
(1-\delta)\int_{f^{-1}(A)} |H(x)|d\sigma_x=
%&=& \\
\int_{f^{-1}(D(R))} |H(x)|d\sigma_x-
\delta \int_{f^{-1}(A)} |H(x)|d\sigma_x$$
\iffalse
\begin{eqnarray*}
\int_{D(R)} |H(x)|d\sigma_x=
%\int_{D(R)\setminus A} |H(x)|d\sigma_x+\int_{A} |H(x)|d\sigma_x=
%&=& \\
\int_{D(R)\setminus A} |T_fH(x)|d\sigma_x
+\int_{A} |T_fH(x)|d\sigma_x&<& \\
\int_{f^{-1}(D(R)\setminus A)} |H(x)|d\sigma_x+
(1-\delta)\int_{f^{-1}(A)} |H(x)|d\sigma_x=
%&=& \\
\int_{f^{-1}(D(R))} |H(x)|d\sigma_x-
\delta \int_{f^{-1}(A)} |H(x)|d\sigma_x,
\end{eqnarray*}
\fi
which contradicts~(\ref{limr2}).
%Hence,
%\begin{equation}\label{tri1}
%|H(x)|=|T_fH(x)|=|\sum_{w: f(w)=x}\frac{H(w)}{f'(w)^2}|=\sum_{w: f(w)=x}\frac{|H(w)|}{|f'(w)|^2}
%\end{equation}
%almost everywhere.

{\bf II}. %Since $H=\hat m$, the Cauchy transform
%of a finite complex measure
%$$m=\sum_{k\ge 0}\alpha_k \delta_{b_k}$$
%and $H$ is integrable on compacts of $\C$,
%then $H$ is well-defined on a set $Y\subset \C$ such that $\C\setminus Y$ has zero Lebesgue measure.
%Moreover, since $m$ is a non-zero measure, $H(x)\neq 0$ on a set $Z$ of positive Lebesgue measure.
By Sect \ref{cauchy}, $H$ is well-defined on a set $Y\subset \C$ such that $\C\setminus Y$ has zero Lebesgue measure and
$H(x)\neq 0$ on a set $Z\subset Y$ of a positive Lebesgue measure. Replacing $Y$ by $\cap_{n\in\Z}f^n(Y)$,
one can assume that $Y$ is completely invariant.
(\ref{tri1}) immediately implies
\begin{lem}\label{iii}
For every measurable $A\subset \C$ such that
$$\Lambda(A):=\int_{A}|H(z)|d\sigma_z<\infty,$$
we have:
$$\Lambda(A)=\Lambda(f^{-1}(A)).$$
In other words, $\Lambda$ is an $f$-invariant positive measure on $\C$ (which is finite in the case (1.1) though not necessary finite in the case (1.2)).
\end{lem}
{\bf III}.
\begin{lem}\label{iv}
Assume $x\in Y$, $f'(x)\neq 0$ and $H(x)\neq 0$. Then there is a real constant $L_x\ge 1$
such that
\begin{equation}\label{3.1}
H(f(x))(f'(x))^2=L_x H(x).
\end{equation}
In particular, $x\in Z$ implies $f(x)\in Z$.
If $f(x),x$ are in $K^c:=\C\setminus K$ then there is a neighborhood $U$ of $x$ so that (\ref{3.1}) holds for all $x\in U$ and, moreover, $L_x$ is a constant function in $U$.
\end{lem}
\begin{proof}
(\ref{3.1}) follows at once from (\ref{tri1}). If, additionally, $f(x),x\in K^c$ then there is a connected neighborhood
$U$ of $x$ such that $U\cup f(U)\subset K^c$ and $H,f'$ do not admit $0$ in $U$. It follows that
$L_x=H(f(x))(f'(x))^2/H(x)$ is a real holomorphic function in $x\in U$, therefore, a constant.
\end{proof}
Let
$$l(z)=\overline{H(z)}/|H(z)|$$
whenever $H(z)$ is well-defined and not zero. Since $L_x>0$ in (\ref{3.1}), it follows that
$$l(f(z))\overline{f'(z)}/f'(z)=l(z)$$
whenever $z\in Z$ and $f'(z)\neq 0$
%$H(z),H(f(z))$ is well-defined and $H(z),f'(z)$ are not zero.
This means that
$l(z)\overline{dz}/dz$ is
an invariant line field defined initially on the set of all $z\in Z$, $f'(z)\neq 0$.
%$z$ so that $H(z),H(f(z))$ are well-defined and $H(z),f'(z)\neq 0$.

%{\bf IV}.
Consider the case $J(f)=\bar\C$. The condition (CL) that $K$ is a C-compact implies that $K$ is nowhere dense, hence, $\C\setminus K$ is a non-empty open set. Assume by contradiction that $H(z)\neq 0$ on a non-empty open $U\subset \C\setminus K$.
Then $l(z)\overline{dz}/dz$ is an invariant holomorphic linefield on $U$. Therefore, by Lemma 3.16 of \cite{mcm},
%[McMullen,Complex Dynamics and Renormalization, pp 49-50],
$f$ is a flexible Lattes map which contradicts to condition (i). Thus $H\equiv 0$ off a C-compact $K$, hence, by (classical) fact of Lemma 5.3, \cite{Le}, $H\equiv 0$, i.e., we are done in this case.

{\bf IV}. From now on, $J(f)\neq \bar\C$.
%, i.e., $F(f):=\bar\C\setminus J(f)$ is a non-empty open set.
To deal with the (non-)integrability of $H$ at $\infty$ we use the following
\begin{lem}\label{paraint}
Let $g$ be a local holomorphic map in a neighborhood of $0$ such that $g(0)=0$, $g'(0)=1$. Let
$h(z)=\tilde A/z^3 + \tilde B/z^2 + \tilde h(z)$  where $\tilde A, \tilde B\in\C$ and $\tilde h$ is an integrable function in a neighborhood of $0$. Assume that either (a) $g(z)=z+O(z^3)$ or (b) $g(z)=z+O(z^2)$ and $\tilde A=0$.
Then:
\begin{itemize}
\item[(1)] every attracting petal $P$ of $g$ at $0$ contains a domain $U_P$ such that $g(U_P)\subset U_P$, $U_P\setminus g(U_P)$ contains a disk, $0\in\partial U_P$, every forward orbit in $P$ enters $U_P$
and
$$\int_{U_P}|h(z)|d\sigma_z<\infty,$$
\item[(2)] there is an open set $U_{-}$ such that $g^{-1}(U_{-})\subset U_{-}$, the union of $U_{-}$ with all attracting petals of $g$ at $0$ and the point $\{0\}$ constitutes a neighborhood of $0$ and
$$\int_{U_{-}}|h(z)|d\sigma_z<\infty.$$
\end{itemize}
\end{lem}
\begin{proof} (1) Begin with a remark that given a local holomorphic injection
(coordinate change) $\psi$ near $0$ such that $\psi(0)=0$ and $\psi'(0)=1$ it is easy to see that enough to show the existence of $\tilde U$ instead of $U$ as in (1) for
$\tilde g=\psi^{-1}\circ g\circ\psi$ and $h\circ\psi$ instead of $g$ and $h$ respectively and then let
$U=\psi(\tilde U)$. Consider the case (a). One can assume that the asymptotic attracting direction of $\tilde U$ is the positive real axis. We use classical facts about local dynamics near a parabolic fixed point,
see e.g., \cite{cg}.
There exists a local coordinate change $\psi$ such that $\tilde g(z)=z-z^{\nu+1}+\alpha z^{2\nu+1}+O(|z|^{2\nu+2})$ where $\nu\ge 2$ (since $g(z)=z+O(z^3)$). Moreover, making a local change of coordinate $w=l(z):=1/(\nu z^\nu)$
in the attracting petal $\tilde P=\psi^{-1}(P)$ of $\tilde g$ we get
$F(w)=l\circ\tilde g\circ l^{-1}(w)=w+1+C/w+O(1/|w|^2)$ as $w\to\infty$ in $P_\infty=l(\tilde P)$ which contains a set of the form $\{w=w_1+iw_2: w_1>M_0-\kappa|w_2|\}$ for some big $M_0>0$ and $\kappa>0$.
Finally, there is a Fatou coordinate $\Psi: U_\infty\to\C$ such that $\Psi(w)=w-C\log w+O(1)$ for the main branch of $\log$ such that $\Psi\circ F(w)=\Psi(w)+1$. Now, given $\epsilon>0$ and $M>0$ let
$U_\infty(\epsilon,M):=\{w=w_1+i w_2: w_1>M, |w_2|<w_1^\epsilon\}$. It is straightforward to check that given $\epsilon>0$ there is $M_\epsilon$ such that $F(U_\infty(\epsilon,M_\epsilon))\subset U_\infty(\epsilon,M_\epsilon)$. Making $M_\epsilon$ even bigger if necessary and using the asymptotics for the Fatou coordinate we check also that
every forward orbit of $F$ in $P_\infty$ enters $U\infty(\epsilon,M_\epsilon)$. Let us fix $\epsilon\in (0,1/\nu)$ and the corresponding $U_\infty:=U_\infty(\epsilon,M_\epsilon)$. Let $\tilde U:=l^{-1}(U_\infty)\subset\tilde P$. Now, if $w=w_1 \pm i w_1^\epsilon\in\partial U_\infty$, for a big $w_1>0$, then $z=l^{-1}(w)=1/(\nu w)^{1/\nu}(w)=u+iv$ where $u=\nu^{-1/\nu} w_1^{-1/\nu}(1+O(w_1^{-2+2\epsilon}))$ and $v=\pm\nu^{-1-1/\nu} w_1^{-1-1/\nu+\epsilon}(1+O(w_1^{-1+\epsilon}))$. Therefore, as $v\to 0$
the following asymptotics holds: $|v|=B u^\gamma+O(|u|^{\gamma'})$
where $\gamma=\nu+1-\epsilon\nu>2$ (as $\nu\ge 2$ and $0<\epsilon<1/\nu$) and $\gamma'=\gamma+\nu(1-\epsilon)>\gamma$. Since $h(\psi(z))=O(|z|^{-3})$ we then get that $\int_{\tilde U}|h(\psi(z))|dxdy<\infty$.
The case (b) is very similar to (a) though simpler and is left to the reader.

Part (2) follows if we apply part (1) to attracting petals of the local inverse $g^{-1}$ finding for each such petal $R_j$, $1\le j\le \nu$, the corresponding to $g^{-1}$ set $U_{R_j}$. Let $U_{-}=\cup_{j=1}^\nu U_{R_j}$. Since each forward orbit in $R_j$ by $g^{-1}$ enters $U_{R_j}$,
the union of $U_{-}$ and the attracting petals of $g$ is a punctured neighborhood of $0$.
\end{proof}

{\bf V}. Let $\Omega$ be a component of $F(f)$.

Recall that $\Lambda$ is a measure introduced in Lemma \ref{iii}, Step II.
%By Sullivan's no wandering domain theorem, $\Omega$ is either periodic or pre-periodic.
\begin{lem}\label{v}
If $\Omega$ is not periodic then $H|_{\Omega}=0$.
\end{lem}
\begin{proof} Assume the contrary and choose $A\subset \Omega$ such that $\mu(A)>0$. Since $\Omega$ is not periodic, $f^{-n}(A)\cap f^{-m}(A)=\emptyset$ for all non-negative $n\neq m$.
Hence, by Step II, $\Lambda(\cup_{n\ge 0}f^{-n}(A))=\int_{\cup_{n\ge 0}f^{-n}(A)}|H(z)|d\sigma_z=\infty$.
In the case (1.1) $H$ is integrable on $\bar\C$, a contradiction.
Consider the case (1.2). If $|\sigma|>1$ then $\infty$ is an attracting fixed point of $f$. Hence, all $f^{-n}(A)$ stay away from a neighborhood of $\infty$, therefore,
$\Lambda(\cup_{n\ge 0}f^{-n}(A))=\int_{\cup_{n\ge 0}f^{-n}(A)}|H(z)|d\sigma_z<\infty$, a contradiction.
In the remaining two possibilities of the case (1.2),
$\infty$ is a parabolic fixed point. Let $\Omega_\infty$ be its immediate basin. As $\cup_{n\ge 0}f^{-n}(A)\subset \C\setminus \Omega_\infty$, we get a contradiction if show that
\begin{equation}\label{parainf}
\int_{\C\setminus \Omega_\infty}|H(z)|d\sigma_z<\infty.
\end{equation}
To this end, passing to $f^q$ we get:
either $f^q(z)=z+O(1/z)$ or $f(z)=z+O(1)$ and $A=0$. Now, making the change $w=1/z$ we arrive at a map
$g(w)=1/f(1/w)$ and $|H(z)|d\sigma_z=|h(w)|d\sigma_w$ where $h(w)=H(1/w)w^{-4}$ precisely as in Lemma~\ref{paraint} of Step IV. Let $U_{-}$ be the set appeared in (2) of that Lemma. Since $\C\setminus\Omega_\infty\cap\{|z|>R\}\subset 1/U_{-}$ for $R$ big enough, (\ref{parainf}) follows.
\end{proof}
{\bf VI}. We are left with the case when $\Omega$ is a periodic component of $F(f)$.
%of some period $q\ge 1$.
\begin{lem}\label{vi}
\begin{enumerate}
\item[(1)] If $\Omega$ is not a Herman ring then $H|_{\Omega}=0$.
\item[(2)] (cf. \cite{Le1}, p.190; \cite{Astorg}) If $\Omega$ is a Herman ring and $H|_{\Omega}\neq 0$ then
there is $C_\Omega\neq 0$ so that,
%for $k=0,1,\cdots,q-1$, $H|_{f^k(A)}=C_a\{\frac{(\varphi_A\circ f^{-k})'}{\varphi_A\circ f^{-k}}\}^2$
$H|_\Omega=C_\Omega\{\frac{\varphi_\Omega'}{\varphi_\Omega}\}^2$ where $\varphi_\Omega: \Omega\to \{1<|w|<R_\Omega\}$ is a conformal isomorphism. Moreover, $C_\Omega=C_{f^i(\Omega)}$ for all $i>0$,i.e. depends only on the cycle that contains $\Omega$. Moreover, $\partial\Omega\subset K$.
%and $f^{-k}: f^k(A)\to A$ is an inverse branch of $f^k$.
%Furthermore, $\partial\Omega\subset K$ and $\varphi_\Omega'\in L^r(dxdy)$ for $2\le r <4$.
\end{enumerate}
\end{lem}
\begin{proof} Let $q$ be so that $f^q(\Omega)=\Omega$ and $H|_\Omega\neq 0$, i.e., $H$ is a non-zero holomorphic function on a connected open set $\Omega$.
First, let $\Omega$ be an immediate basin of attraction of either attracting or parabolic point $a\in\overline\Omega$. To prove that $H|_\Omega=0$ it is enough to find an open set $X\subset\Omega$ such that $\Lambda(X)<\infty$ and either (i) $X\subset f^{-q}(X)$, $f^{-q}(X)\setminus X$ contains a ball or (ii)
$X\supset f^{-q}(X)$, $X\setminus f^{-q}(X)$ contains a ball. Indeed, then
$\Lambda(f^{-q}(X)\setminus X)=\Lambda(f^{-q}(X))-\Lambda(X)=0$ in case (i) and $\Lambda(X\setminus f^{-q}(X))=0$ in case (ii), hence, in either case $H=0$ on a ball in $\Omega$, hence, everywhere in $\Omega$. Let us show that such a set $X$ exists. If $a$ is attracting or parabolic and $a\neq\infty$, $X$ can be taken a neighborhood of $a$ if $a$ attracting and an attracting petal at $a$ if $a$ parabolic. If $a=\infty$ attracting (i.e., $|\sigma|>1$), define $X=\Omega\setminus U$ where $U$ is a neighborhood of $\infty$. Finally, if $a=\infty$ parabolic, by Lemma~\ref{paraint} define $X=1/U_P$ where $P$ is an attracting petal at $0$ of
$g(z)=1/f(1/z)$.

Now, let $\Omega$ be either a Siegel disk or a Herman ring. Since $f^q:\Omega\to\Omega$ is a homeomorphism, by Step V, $H$ is a holomorphic non-zero map on $\Omega$ such that $H\circ f^q [(f^q)']^2=H$ on $\Omega$.
Let $\varphi_\Omega:\Omega\to \Delta$ be a conformal homeomorphism onto either a disk (if $\Omega$ is a Siegel disk) or an annulus $\Delta$ (if $\Omega$ is a Herman ring) which conjugates $f^q:\Omega\to\Omega$ with
an irrational rotation $w\mapsto\lambda w$ on $\Delta$. For $\tilde H=H\circ\varphi_\Omega^{-1} [(\varphi_\Omega^{-1})']]^2$, the equation for $H$ turns into: $\tilde H(\lambda w)\lambda^2=\tilde H(w)$, $w\in \Delta$. Writing this equation in
terms of a series $\tilde H(w)=\sum_{n=-\infty}^\infty a_n w^n$, it is immediate that the only solution is $\tilde H(w)=a_{-2}/w^2$, that is, the
case of a disk (when $a_n=0$ for $n<0$) is impossible, while if $\Delta$ is an annulus,
$H$ must be proportional to $\{\frac{\varphi_\Omega'}{\varphi_\Omega}\}^2$. Moreover, since $H|_\Omega\neq 0$, every point $z\in\partial\Omega$ must belong to $K$. Indeed, otherwise
$H$ is a holomorphic function in a neighborhood $V$ of $z$ such that $H\neq 0$ in $\Omega\cap V$. On the other hand $z=\lim_{n\to\infty}z_n$ for a sequence $z_n$ of points of some non-periodic components of $F(f)$,
hence, by Lemma~\ref{v}, $H(z_n)=0$ and by the Uniqueness Theorem, $H|_V=0$, a contradiction.
%Finally, $H=C_\Omega\{\varphi_\Omega'/\varphi_\Omega\}^2\in L^t(dxdy)$ for $1\le t <2$
%because $H$ is a Cauchy transform of a finite measure, hence, is locally in $L^t(dxdy)$ for $1\le t<2$.
\end{proof}
By Steps V-VI, $H=0$ outside $K\cup\mathcal{H}_K$.

{\bf VII}. $\mathcal{H}(K)=\emptyset$, i.e., $K$ contains no boundaries of Herman rings. Then by Steps V-VI, $H=\hat\mu=0$
off $K$. Assume (CL), i.e., $K$ is a C-compact. Then, by (well-known) Lemma 5.3 of \cite{Le}, $\mu=0$,
%i.e., $H=0$ everywhere,
%the measure $m=\sum_{k\ge 0}\alpha_k\delta_{b_k}=0$,
a contradiction.

{\bf VIII}: $\mathcal{H}(K)\neq\emptyset$. Assuming (CL), i.e., $K$ is a C-compact, by Lemma \ref{remov}, $\mu=0$ in $\C\setminus\mathcal{H}_K$, i.e., $K=supp(\mu)\subset\partial\mathcal{H}_K$
(cf. with the proof of Theorem 1 of \cite{lll}).
By Lemma \ref{vi}(2), for each $A\in\mathcal{H}_K$,
$H|_A=\hat\mu|_A=C_A\{\varphi_A'/\varphi_A\}^2$, therefore,
\begin{equation}\label{viii}
\hat\mu\circ\psi_A(w) \psi_A'(w)=C_A\{w^2\psi_A'(w)\}^{-1}.
\end{equation}
%Now, as $K$ is nowhere dense, $\hat K:=K\cup\mathcal{H}_K$ is a proper compact subset of $\overline\C$.
Now we also assume (AL). There are two cases. If $\mathcal{H}_K$ is bounded, we directly apply Proposition \ref{rings} to $E=K$, $\mathcal{H}=\mathcal{H}(K)$, and $\nu=\mu$, and get the desired conclusion. If $\infty\in\mathcal{H}_K$, let $M(z)=1/(x_0-z)$ for some $x_0\in\C\setminus\mathcal{H}_K$ and let
$E=M(K)$, $\mathcal{H}=\{M(A)| A\in\mathcal{H}_K\}$, and $\nu$ is defined by
$d\nu(u)=ud\mu(M^{-1}(u))$ so that $\hat\nu(v)=v^{-1}\hat\mu(M^{-1}(v))=0$ off the bounded set $E\cup V\subset\C$. Now we can apply Proposition \ref{rings} to the rational map $M\circ f\circ M^{-1}$,
its collection of bounded Herman rings $\mathcal{H}=\{B=M(A)| A\in\mathcal{H}_K\}$, the set $E$ and the measure $\nu$ just defined. We obtain the representation $\nu=\sum_{\{B\in\mathcal{H}\}}\nu_B$ where
$\nu_B$ is supported on $\partial B$ and absolutely continuous w.r.t. the harmonic measure
$\omega_B$ of $B$, i.e., $d\nu_B=h_B d\omega_B$ where $h_B\in L^1(\partial B, \omega_B)$.
By the connection of $\mu$ and $\nu$ and since $d\omega_{M(a)}(M(z))=M'(z)d\omega_A(z)$, we get the
representation (\ref{repr}), as in the conclusion of part Lemma \ref{Tcauchy}, part {\bf 1}.
That $1/\psi_A'\in H^1$ provided $C_A\neq 0$
%If all $A$ are bounded, (\ref{hardy})
follows from (\ref{viii}) as well.

Part {\bf 1} has been proved.

Let us prove Part {\bf 2}.  After perhaps a Mobius change, one can assume that all Herman rings $B\in\mathcal{H}$ are bounded subsets of $\C$.
Let us show that $1/\psi_B'\in H^1$ for all $B\in\mathcal{H}$. By a condition this holds for $B=A$. Let now $B=f^j(A)$ for some $j\in\{1,\cdots,q-1\}$. Then $\psi_B=f^j\circ\psi_A:\Delta_A\to B$ and, for an appropriate $\lambda\in S^1$ and all
$w\in\Delta_A$,
$$|\psi_B'(w)|=|(f^j)'(\psi_A(w))\psi_A'(w)|=|\frac{\psi_A'(w)(f^q)'(\psi_A(w))}{(f^{q-j})'(f^j(\psi_A(w)))|}=$$
$$\frac{|\lambda\psi_A'(\lambda w)|}{|(f^{q-j})'(f^j(\psi_A(w)))|}\ge M^{-(q-j)}|\psi_A'(\lambda w)|$$ where
$M=\sup\{|(f'(z)|: z\in\cup_{B\in\mathcal{H}}B\}<\infty$. We get immediately that $1/\psi_B'\in H^1$ as well.
Now the existence of the measure $\mu$
follows easily from (\ref{viii}) if we apply \cite{lll}, Theorem 1,{\bf P2} where we take
$\Omega_i=f^{i-1}(A)$, $i=1,\cdots,q-1$, and $\kappa_i=(\varphi_i'/\varphi_i)^2$
where $\varphi_i:f^{i-1}(A)\to \{1<|w|<R\}$ is a conformal homeomorphism. As for the uniqueness,
if $\nu$ is another measure as in Part {\bf 2}, by (\ref{viii}), there is $C\in\C$ such that, for the measure $\tau:=\nu-C\mu$, $\hat\tau=0$ off $K:=\cup\{\partial B: B\in\mathcal{H}\}$. On the other hand,
$K$ is a C-compact because every $x\in K$ lies at the boundary of one of the components $B\in\mathcal{H}$
of its complement. Hence, $\tau=0$.

%\begin{theo}
%Suppose $f$ is as in the Main Theorem though we don't put any restriction on the compact $K$, but ASSUME that (b) is not the case for $f$ (for example, $f$ satisfies one of the condition (1)-(7) of the previous Comment~\ref{c1} where one should take $K=\cup_{j=1}^r \omega(c_j)$). Then either the conclusion of the Main Theorem remains valid or $K$ has a positive area and carries a non-trivial invariant linefield and, therefore,
%there is a non-trivial family of rational maps through $f$ so that each of them is conjugate to $f$
%by a quasi-conformal homeomorphism which is conformal off $J(f)$.
%\end{theo}
%\begin{proof} If $K$ has zero area, then $K$ is a C-compact and its conclusion holds.
%Now, assume $K$ has positive area. Since $H$ is the Cauchy transform of a non-trivial discrete  measure and $H=0$ off $K$, we conclude that $H$ is non-zero a.e. on $K$. Then $\overline{H(z)}/|H(z)|d\bar{z}/dz$ defines an invariant linefield a.e. on $K$, see (IV).
%\end{proof}

{\it Case (8) of Corollary \ref{c1}}. So assume that every critical point in $J(f)$ is summable.
If $J(f)\neq\bar\C$, then by \cite{R-L}, $J(f)$ is of measure zero and $f$ has no Herman rings, hence,
$K\subset J(f)$ is a C-compact and $\mathcal{H}_K$ is empty, and Lemma \ref{Tcauchy} applies.
Let $J(f)=\bar\C$, in particular, $f$ has no Herman rings. Assume $H$ is non-trivial. By Step III of the above proof of Lemma \ref{Tcauchy}, $f$ admits a non-trivial invariant line field on a forward invariant set of a positive Lebesgue measure (which is the set $Z$ minus forward orbits of critical points).
%If $K$ is of area zero, $K$ is a C-compact, and we are done again. Let $K$ is of a positive area (for example, $K=\bar\C$). By Step III of the proof of Lemma \ref{Tcauchy},
%assuming $H$ is non-trivial????? $f$ admits a non-trivial line field on $K$.
On the other hand, by \cite{RS} (see also \cite{Pr}), Lebesgue almost every point of $\C$ is conical for $f$, which leads (as in \cite{mcm}) to the existence of a holomorphic line field for $f$, hence, $f$ is a flexible Lattes example,a contradiction.

\subsection*{Acknowledgements} We would like to thank Magnus Aspenberg for asking about the proof of Lemma 5.2 in \cite{Le},
Matthieu Astorg for discussions, and the referee for careful reading of the paper and helpful comments.

\end{document}